\documentclass[12pt]{amsart}
\usepackage{amsmath,amssymb,amsfonts,amsthm,amsopn}
\usepackage{graphics}

\newtheorem{tm}{Theorem}[section]
\newtheorem{lemma}[tm]{Lemma}
\newtheorem{prop}[tm]{Proposition}

\newtheorem{theorem}{Theorem}[section]
\newtheorem{corollary}[theorem]{Corollary}

\newtheorem{example}[theorem]{Example}

\newtheorem{proposition}[theorem]{Proposition}
\newtheorem{remark}[theorem]{Remark}

\newcommand{\beqa}{\begin{eqnarray*}}
	\newcommand{\eeqa}{\end{eqnarray*}}

\DeclareMathOperator*{\supp}{supp}

\newcommand{\field}[1]{\mathbb{#1}}
\newcommand{\bR}{\field{R}}        
\newcommand{\bN}{\field{N}}        
\def\la{\lambda}

\def\eps{\epsilon}

\def\cF{\mathcal{F}}              
\def\cS{\mathcal{S}}

\def\cC{\mathcal{C}}

\def\a{\aleph}

\def\rd{\bR^d}

\def\rdd{{\bR^{2d}}}

\def\lrd{L^2(\rd)}

\def\intrdd{\int_{\rdd}}

\def\R{\right)}

\def\<{\left<}
\def\>{\right>}

\def\inv{^{-1}}

\def\mv1{M_v^1}

\def\a{\alpha}
\def\b{\beta}

\def\R{\mathbb{R}}
\def\Ren{\mathbb{R}^d}

\def\sch{\mathcal{S}}

\def\Fur{\mathcal{F}}

\def\f{\varphi}

\def\Sn2{S_{2}(L^{2}(\Ren))}
\def\S1{S_{1}(L^{2}(\Ren))}
\def\sig00{\sigma_{0,0}}

\def\la{\langle}
\def\ra{\rangle}

\def\p{\Phi}
\def\b{\beta}
\def\g{\gamma}

\begin{document}

\begin{abstract}
We study continuity properties in Lebesgue spaces for a class of Fourier integral operators arising in the study of the Boltzmann equation. The phase has a H\"older-type singularity at the origin. We prove boundedness in $L^1$ with a precise loss of decay depending on the H\"older exponent, and we show by counterxemples that a loss occurs even in the case of smooth phases. The results can be seen as a quantitative version of the Beurling-Helson theorem for changes of variables with a  H\"older singularity at the origin. The continuity in $L^2$ is studied as well by providing sufficient conditions and relevant counterexemples. The proofs rely on techniques from Time-frequency Analysis. 
\end{abstract}
\title{On Fourier integral operators with H\"older-continuous phase}\author{Elena Cordero, Fabio Nicola and Eva Primo}
\address{Department of Mathematics,
University of Torino, via Carlo Alberto
10, 10123 Torino, Italy}
\address{Dipartimento di Matematica,
Politecnico di Torino, corso Duca degli
Abruzzi 24, 10129 Torino, Italy}
\address{Departament d'An\`alisi Matem\`atica, Universitat de Val\`encia, Dr. Moliner 50, 46100-Burjassot, Val\`encia (Spain)}
\email{elena.cordero@unito.it}
\email{fabio.nicola@polito.it}
\email{eva.primo@uv.es}
\thanks{}
\maketitle

\section{Introduction}
In the study of the Boltzmann equation one faces with integral operators (Boltzmann collision operators) \begin{equation}\label{IO}
A f(x)= \int_{\bR^d} K(x,y)f(y)dy,\quad f\in \cS(\rd),
\end{equation} with  (collision) kernels 
\begin{equation}\label{K}
K(x,y)= \int_{\bR^d}\Phi(u) e^{- 2 \pi i  ( \b(|u|)u \cdot y -u \cdot x)}du,
\end{equation}
and is interested in estimates of the type 
\begin{equation}\label{es0}
\sup_{y\in\rd}\int_{\rd}|K(x,y)|\, dx<\infty,
\end{equation}
\cite{Lods2017}  (related references are provided by \cite{Lod2010, BD1998, MR2264623}). The estimate \eqref{es0} would imply the corresponding operator $A$ is bounded on $L^1(\rd)$. The function $\Phi(u)$ has a good decay at infinity but could be not smooth at the origin $u=0$. A typical example is given by radial functions 
\begin{equation}\label{PHIex}
\Phi(u)=\frac{|u|}{(1+|u|^2)^m}
\end{equation}
with large real $m$.\par
The phase $\beta(r)$ is real-valued and smooth on $(0,+\infty)$ but could have a H\"older type singularity at the origin. As an oversimplified model the reader can consider the case 
\begin{equation}\label{betaintro}
\beta(r)=a+br^\gamma,\quad  0<r\leq 1,
\end{equation}
for some $a,b\in\R$, $\gamma\in (0,1)$.  As $r\to+\infty$, $\beta(r)$ is assumed to approach a constant. \par
As basic case suppose $\beta(r)=a$, $r>0$, is a constant function. Rapid granular flows are described by
the Boltzmann equation and $\beta(r)=a$ corresponds to the case of inelastic interactions with constant restitution coefficient. 
Indeed, the loss of mechanical energy due to collisions is characterized by the  restitution coefficient $\beta$ which quantifies the loss of relative normal velocity of a pair of colliding particles after the collision with respect to the impact velocity.
Now, when $\beta(r)=a$ is constant,
 $$K(x,y)=\widehat{\Phi}(ay-x)$$ and the estimate \eqref{es0} holds if and only if $\Phi\in \mathcal{F} L^1(\rd)$, i.e.\ $\Phi$ has Fourier transform in $L^1(\rd)$. 
The major part of
the research, at the physical as well as the mathematical levels, has been devoted to this particular case of a constant restitution coefficient. However, as described in \cite{MR2101911, Lod2010}, a more relevant description of granular gases should
involve a variable restitution coefficient $\beta(r)$.

In the model case above $\beta(r)$ approaches a constant both as $r\to 0^+$ and $r\to+\infty$ and is smooth in between, so that one could conjecture that the same estimate holds in that case.  Now, this is not the case, even for smooth phases: we will prove in Proposition \ref{pro3.1} that, in dimension $d=1$, if $\tilde{\varphi}(u):=\beta(|u|)u$ is any {\it nonlinear} smooth diffeomorphism $\R\to\R$ with ${\tilde{\varphi}}(u)=u$ (hence $\beta(|u|)=1$) for $|u|\geq 1$, and $\Phi\in \cC^\infty_0(\R)$, $\Phi\equiv1$ on $[-1,1]$, then the weighted estimate
\begin{equation}\label{es0w}
\int_{\rd}|K(x,y)|\, dx\lesssim(1+|y|)^s
\end{equation}
does not hold for $s<1/2$. \par
This looks surprising at first glance, but it can be regarded as a manifestation of the Beurling-Helson phenomenon \cite{beurling53,GobalFIO,lebedev94,kasso07,MR2844679},  which roughly speaking states that the change-of-variable operator $f\mapsto f\circ \psi$ is not bounded on $\cF L^1(\rd)$ except in the case $\psi:\rd\to\rd$ is an affine mapping. Indeed the operator $A$ in \eqref{IO} with kernel $K(x,y)$ in \eqref{K} can be written as 
\[
Af =\cF^{-1}\Phi\ast \cF^{-1}(\cF f \circ \tilde{\varphi}),\quad{\rm with}\ \tilde{\varphi}(u):=\beta(|u|)u.
\]
Now, one is then interested to the precise growth in \eqref{es0w}. Let us summarize here the main results in the special case of our oversimplified model above. \par\medskip
$\bullet$ {\it Suppose $\beta(r)$ as in \eqref{betaintro} for $0<r\leq 1$, with $\gamma\in (-1,1]$ and assume $\beta$ has at most linear growth as $r\to+\infty$. Let $\Phi$ be as in \eqref{PHIex}, with $m>(d+1)/2$. Then \eqref{es0w} holds with $s=d/(\gamma+1)$.}\par\medskip
As expected, the growth in \eqref{es0w} is therefore the weakest one when $\gamma=1$, being $s=d/2$ in that case. The same growth occurs for smooth phases, as the following result shows. 
\par\medskip
$\bullet$ {\it Suppose that $\tilde{\varphi}(u):=\beta(|u|)u$ extends to a smooth function in $\rd$, with an at most quadratic growth at infinity.  Let $\Phi$ be as in \eqref{PHIex}, with $m>(d+1)/2$. Then \eqref{es0w} holds with $s=d/2$.}\par\medskip
By the above mentioned counterexample, this estimate is sharp, at least in dimension $d=1$. \par
Notice that the estimate \eqref{es0w} implies a continuity property for the corresponding operator between weighted $L^1$ spaces, precisely $L^1_{v_s}\to L^1$, where $v_s(x)=(1+|x|)^s$.\par
A natural question is therefore whether similar continuity estimates hold without a loss of decay at least in $L^2(\rd)$, under the above assumptions. We will show in Proposition \ref{41} below that, again, this is not the case. Sufficient conditions are instead given in \ref{tl2cont} below. Here is a simplified version of Theorem \ref{tl2cont} (and subsequent remark).  
\par\medskip
$\bullet$ {\it Suppose $\beta(r)$ as in \eqref{betaintro} for $0<r\leq 1$, with $\gamma>0$. Let $\Phi\in \cC^\infty(\rd)$ supported in $|u|\leq 1$.  Then, if $a(a+(\gamma+1)b)>0$ the operator $A$ in \eqref{IO}, \eqref{K} is bounded in $L^2(\rd)$.}\par\medskip
Actually the results below are stated for $\beta$ and $\Phi$ in classes of functions with minimal regularity and are inspired by the models above. It turns out that, in all the results it is sufficient to take $\Phi$ in the so-called Segal algebra $M^1(\rd)$ \cite{Segal81,Feich89,Feich2006}. Roughly speaking, a function $\Phi\in L^\infty(\rd)$ belongs to  $M^1(\rd)$ if locally has the regularity of a function in $\cF L^1(\rd)$ (in particular is continuous) and globally it decays as a function in $L^1(\rd)$,  but no differentiability conditions are required. We have $M^1(\rd)\hookrightarrow L^1(\rd)\cap \cF L^1(\rd)$. To compare this space with the usual Sobolev spaces we observe that $W^{k,1}(\rd)\subset M^1(\rd)$ for $k\geq d+1$, but functions in $M^1(\rd)$ do not need to have any derivatives. For examples, the functions $\Phi$ in \eqref{PHIex} are in $M^1(\rd)$ if $m>(d+1)/2$ (see Example \ref{exm1} below). It is important to observe that the very weak assumption $\Phi\in M^1(\rd)$ prevents us to use classical tools such as stationary phase estimates; instead we use techniques and function spaces from Time-frequency Analysis, which will be recalled in the next Section. Recently, such function spaces and more generally Time-frequency Analysis have been successfully applied in the study of partial differential equations with rough data  by a large number of authors, see, e.g., \cite{CNStricharzJDE2008,Ruzhansky2016,SugimotoWang2011,Wang2007} and references therein. 
We also refer to the papers \cite{GobalFIO,MR3165647} and the references therein for the problem of the continuity in $L^p(\rd)$, $1<p<\infty$, and from the Hardy space to $L^1(\rd)$, of general Fourier integral operators of H\"ormander's type (i.e. arising in the study of hyperbolic equations). \par\medskip

In short the paper is organized as follows. \\
In Section \ref{sec1} we briefly recall the  definitions of modulation and Wiener amalgam spaces and exhibit the main properties and preliminary results we need in the sequel.\par

In Section \ref{contl1} we study  the $L^1$-continuity  for the integral operators in \eqref{IO} having phases with H\"{o}lder-type singularity at the origin. The boundedness is attained at the cost of a loss of decay. Such a loss is unavoidable, as testified by an  example in dimension $d=1$ (cf. Proposition \ref{pro3.1}).\par 

In Section \ref{contl2} we study the $L^2$-continuity properties of $A$ in \eqref{IO}. Under the same assumptions of the $L^1$-boundedness results we provide a counterexample even in this framework (cf.\ Proposition \ref{41}). We then show conditions on the phase of the operators which guarantee $L^2$-boundedness  without loss of decay.

\section{Preliminaries}\label{sec1}



\textbf{Notation.} We define
$|x|^2=x\cdot x$, for
$x\in\rd$, where $x\cdot
y=xy$ is the scalar product
on $\rd$. The space of
smooth functions with compact
support is denoted by
$\cC_0^\infty(\rd)$, the
Schwartz class is
$\sch(\rd)$, the space of
tempered distributions
$\sch'(\rd)$.    The Fourier
transform is normalized to be
${\hat
	{f}}(u)=\cF f(u)=\int
f(t)e^{-2\pi i tu}dt$.
Translation and modulation operators ({\it time and frequency shifts}) are defined, respectively, by
$$ T_xf(t)=f(t-x)\quad{\rm and}\quad M_{u}f(t)= e^{2\pi i u
	t}f(t).$$
We have the formulas
$(T_xf)\hat{} = M_{-x}{\hat
	{f}}$, $(M_{u}f)\hat{}
=T_{u}{\hat {f}}$, and
$M_{u}T_x=e^{2\pi i
	xu}T_xM_{u}$. The notation
$A\lesssim B$ means $A\leq c
B$ for a suitable constant
$c>0$, whereas $A \asymp B$
means $c\inv A \leq B \leq c
A$, for some $c\geq 1$. The
symbol $B_1 \hookrightarrow
B_2$ denotes the continuous
embedding of the linear space
$B_1$ into $B_2$.

\subsection{Wiener amalgam
spaces \cite{Feichtinger_1983_Banach,Feichtinger_1981_Banach,Feichtinger_1990_Generalized,Fournier_1985_Amalgams,Feichtinger_1998_Banach}.}
Let $g \in \cC_0^\infty(\rd)$ be a test function that satisfies $\|g\|_{L^2}=1$. We will refer to $g$ as a window function. For $1\leq p\leq \infty$, recall the $\cF L^p$ spaces, defined by $$\cF L^p(\R^d)=\{f\in\cS'(\R^d)\,:\, \exists \,h\in L^p(\R^d),\,\hat h=f\};
$$
they are Banach spaces equipped with the norm
\begin{equation}\label{flp}
\| f\|_{\cF L^p}=\|h\|_{L^p},\quad\mbox{with} \ \hat h=f.
\end{equation}
Let $B$ one of the following Banach spaces: $L^p, \cF L^p$, $1\leq p\leq
\infty$. For any given function $f$ which is locally in $B$ (i.e. $g f\in B$, $\forall g\in\cC_0^\infty(\rd)$), we set $f_B(x)=\| fT_x g\|_B$.
The {\it Wiener amalgam space} $W(B,L^p)(\rd)$ with local component $B$ and global component  $L^p(\rd)$, $1\leq p\leq \infty$, is defined as the space of all functions $f$
locally in $B$ such that $f_B\in L^p(\rd)$. Endowed with the norm
$\|f\|_{W(B,L^p)}=\|f_B\|_{L^p(\rd)}$, $W(B,L^p)$ is a Banach space. Moreover,
different choices of $g\in \cC_0^\infty(\rd)$  generate the same space
and yield equivalent norms.

If  $B=\cF L^1$, the Fourier algebra,  the space of admissible
windows for the Wiener amalgam spaces $W(\cF L^1,L^p)(\rd)$ can be
enlarged to the so-called Feichtinger algebra $M^1(\rd)=W(\cF L^1,L^1)(\rd)$, which is also a modulation space, as shown below.
Recall  that the Schwartz class $\cS(\rd)$ is dense in $W(\cF L^1,L^1)(\rd)$.

\subsection{Modulation
spaces \cite{Grochenig_2001_Foundations}.}
Let $g\in\cS(\rd)$ be a non-zero
window function. The
short-time Fourier transform
(STFT) $V_gf$ of a
function/tempered
distribution $f$ with respect
to the window $g$ is
defined by
\[
V_g f(z,u)=\int e^{-2\pi i  u y}f(y)g(y-z)\,dy,
\]
i.e.,  the  Fourier transform $\cF$ applied to $fT_zg$.\par For
$1\leq p, q\leq\infty$, the modulation space $M^{p,q}(\R^d)$ is
defined as the space of measurable functions $f$ on $\R^d$ such that
the norm
\[
\|f\|_{M^{p,q}}=\|\|V_gf(\cdot,u) \|_{L^p}\|_{L^q_u}
\] 
is
finite. Among the properties of modulation spaces, we record that
$M^{2,2}(\rd)=L^2(\rd)$, $M^{p_1,q_1}(\rd)\hookrightarrow M^{p_2,q_2}(\rd)$, if $p_1\leq
p_2$ and $q_1\leq q_2$. If $p,q<\infty$, then
$(M^{p,q}(\rd))'=M^{p',q'}(\rd)$.\par For comparison, notice that the norm in
the Wiener amalgam spaces $W(\cF L^p,L^q)(\rd)$ reads
\begin{equation}\label{normwiener}
\|f\|_{W(\cF L^p,L^q)}=\|\|V_gf(z,\cdot) \|_{L^p}\|_{L^q_z}.
\end{equation}
The relationship between modulation and Wiener amalgam spaces is
expressed by  the following result.
\begin{proposition}\label{vm} For $1\leq p,q\leq\infty$, the
Fourier transform establishes
an isomorphism  $\cF:
M^{p,q}(\rd)\to W(\cF
L^p,L^q)(\rd)$.\par
\end{proposition}
Consequently, convolution properties of modulation spaces can be
translated into point-wise multiplication properties of Wiener
amalgam spaces, as shown below.
\begin{proposition}\label{p3}
For every $1\leq
p,q\leq\infty$ we have
\[
\|fu\|_{W(\cF L^p,L^q)}\lesssim
\|f\|_{W(\cF
L^1,L^\infty)}\|u\|_{W(\cF L^p,L^q)}.
\]
If $p=q$, we have
\begin{equation}\label{ap}
\|fu\|_{M^p}\lesssim
\|f\|_{W(\cF
L^1,L^\infty)}\|u\|_{M^p}.
\end{equation}
\end{proposition}
\begin{proof}
From Proposition \ref{vm}, the estimate to prove is equivalent to
\[
\|\hat{f}\ast
\hat{u}\|_{M^{p,q}}\lesssim
\|\hat{f}\|_{M^{1,\infty}}\|\hat{u}
\|_{M^{p,q}},
\]
but this a special case  of \cite[Proposition 2.4]{Cordero_2003_Time}.
\end{proof}

Modulation and Wiener amalgam spaces are invariant with respect to modulation and translation operators. Namely, from \cite[Theorem 11.3.5 ]{Grochenig_2001_Foundations} we infer

\begin{prop} For $1\leq p,q\leq \infty$, $M^{p,q}(\rd)$ is invariant under time-frequency shifts, with
	\begin{equation}\label{tfmod}
	\|T_x M_u f\|_{M^{p,q}}\asymp \|f\|_{M^{p,q}}.
	\end{equation}
\end{prop}
The result is known for Wiener amalgam spaces as well.
\begin{prop} For $1\leq p,q\leq \infty$, $W(\cF L^p, L^q)(\rd)$ is invariant under time-frequency shifts, with
	\begin{equation}\label{tfwiener}
	\|T_x M_u f\|_{W(\cF L^p, L^q)}\asymp\|f\|_{W(\cF L^p, L^q)}.
	\end{equation}
\end{prop}
Indeed, both these results are a consequence of the covariance property of STFT, namely
\[
|V_g(T_x M_u f)(y,\omega)|=|V_gf(y-x,\omega-u)|,\quad x,y,u,\omega\in\rd,
\]
which follows from direct inspection. \par

\subsection{Preliminary results}
In the sequel we shall list  issues  preparing for  our later argumentation.
To study the properties of our phase function, we shall relay on the following results.
\begin{lemma} (\cite[Lemma 3.2]{Miyachi_2009_Estimates})\label{l1}  Let $\eps>0 $. Suppose $\mu$ is a real-valued function of class $\cC^{\lfloor d/2\rfloor+1}$ on $\R^d \setminus \{0\}$ satisfying \begin{equation}\label{l1est}
	 |\partial^{\alpha}\mu(u)|\leq C_{\alpha}|u|^{\eps - |\alpha|}
	\end{equation} 
	for $|\alpha|\leq \lfloor d/2\rfloor+1$. Then $\cF^{-1}[\eta e^{i  \mu}]\in L^1(\R^d)$ for each $\eta \in \cS(\R^d)$ with compact support. The norm of $\eta e^{i  \mu}$ in $\cF L^1(\R^d)$ is indeed controlled by a constant depending only on $d$, $\eta$ and the constants $C_\alpha$ in \eqref{l1est}.  
\end{lemma}

\begin{lemma}(\cite[Theorem 5]{Benyi_2007_Unimodular}) \label{l2}
For $d\geq 1$, let $l= \lfloor d/2\rfloor+1$. Assume that $\mu$ is $2l$-times continuously differentiable function on $\rd$ and $\|\partial^{\a}\mu\|_{L^{\infty}} \leq C_{\a}$, for $2\leq |\a| \leq 2l$, and some constants $C_{\a}$. Then $e^{i \mu}\in W(\cF L^1, L^{\infty})(\rd)$. \par The norm of $e^{i  \mu}$ in $W(\cF L^1, L^{\infty})(\rd)$ is indeed controlled by a constant depending only on $d$, and the above constants $C_\alpha$.  
\end{lemma}

Dilation properties for Wiener amalgam spaces will play a key role in the proof of our main results.
\begin{lemma}(\cite[Corollary 3.2]{Cordero_2008_Metaplectic})\label{l3}
Let $1 \leq p,q \leq \infty$ and $\lambda\geq 1$. Then, for every $f \in W(\cF L^p, L^q)(\bR^d)$, 
$$
\|f_\lambda\|_{W(\cF L^p, L^q)} \lesssim \lambda^{d/p-d/q} \|f\|_{W(\cF L^p, L^q)},
$$
where $f_\lambda(x)=f(\lambda x)$.
\end{lemma}

\begin{lemma}(\cite[Proposition 2.5]{Cordero_2015_Propagation})\label{l4}
Let $h\in \cC^{\infty}(\bR^d \setminus \{0\})$ be positively homogeneous of degree $r>0$, i.e., $h(\lambda x)= \lambda^r h(x)$ for $x \neq 0$, $ \lambda >0$. Consider $\chi \in \cC^{\infty}_{0}(\bR^d)$ and set  $f=h \chi$. Then, for $\psi \in \cS(\bR^d)$, there exists a constant $C>0$ such that 
$$
|V_{\psi}f(x, u)| \leq C(1 + |u|)^{-r-d}, \,\,\text{for \,\,every    }\,x,u \in \bR^d.
$$
\end{lemma}
We recall the definition, for $k\in\bN$, of the Sobolev space 
$$ W^{k,1}(\bR^d )=\left\{u\in L^{1}(\bR^d):\partial^{\alpha }u\in L^{1}(\bR^d ),\,\,\forall |\alpha |\leq k\right\}$$
with the obvious norm. 
Sharp inclusion relations between Bessel potential Sobolev spaces and modulation spaces were proved in \cite{Wang2006,Kasso2004,Kobayashi_2011_inclusion,Sugimoto-Tomita2007,Toft1-2004}. Here we are interested in an  easier embedding of the Sobolev space $ W^{k,1}(\bR^d )$. We remark that for $p=1$ (and $p=\infty$) Sobolev and Bessel potential Sobolev spaces are different (cf. \cite[pag. 160]{Stein70}).
\begin{lemma}\label{l5}
We have
 $$W^{k,1}(\bR^d) \hookrightarrow M^{1}(\bR^d) \quad\text{for}\  k\geq d+1.$$
\end{lemma}
\begin{proof}
First of all we observe that the following embedding holds:
\begin{equation}\label{wk1}
W^{k,1}(\rd)\hookrightarrow \cF L^1(\rd), \quad k\geq d+1.
\end{equation}
This follows by an integration by parts, namely, for $f\in\cS(\rd)$ and every $\alpha$,
\[
|\Fur f(u)|\leq (2\pi)^{-|\alpha|} |u^\alpha|^{-1}\|\partial^\alpha f\|_{L^1(\rd)}.
\]
Taking the minimum with respect to $|\alpha|\leq d+1$ and using the fact that $\min_{|\alpha|\leq d+1} |u^\alpha|^{-1}$ is in $L^1(\rd)$ (see e.g.\ \cite[page 321]{Grochenig_2001_Foundations}) we obtain \eqref{wk1}.\par
We can then write
\[
\|\Fur (f T_xg)\|_{L^1}\lesssim \sum_{|\alpha|\leq d+1}\|\partial^\alpha(f T_x g)\|_{L^1}.
\]
Using Leibniz' rule and integrating with respect to $x\in\rd$ gives
\[
\|f\|_{M^1}\lesssim \|f\|_{W^{k,1}(\rd)}, \quad k\geq d+1.
\]
\end{proof}

In order to exhibit the counterexample anticipated in the introduction we will make use of a result proved in \cite[Proposition 6.1]{GobalFIO} which can be stated as follows.
\begin{proposition}\label{casolp}
	Let $\tilde{\f}:\R\to\R$ be any nonlinear smooth diffeomorphism 
	satisfying 
	\begin{equation}\label{bound}\tilde{\f}(u)=u,\quad \mbox{for}
	\,\,\,|u|\geq 1,
	\end{equation}
	and let 
	$\Phi\in \cC^\infty_0(\rd)$, $\Phi\equiv1$ on $[-1,1]^d$.\par
	For $2\leq p\leq\infty$,
	$m< d(1/2-1/p)$, 	
	the so-called type I FIO $T_{I,\f,\sigma}$, defined as
	\begin{equation}\label{FIO}
T_{I,\f,\sigma}f(x)=\int e^{2\pi
	i\f(x,u)}
\sigma(x,u)\hat{f}(u)\,du,
\end{equation}
	having phase
	$\f(x,u)=\sum_{k=1}^d\tilde{\f}(u_k)x_k$, and
	symbol $\sigma(x,u)=\langle x\rangle^m\Phi(u)$,
	 does not
	extend to a bounded operator on $L^p(\rd)$.
\end{proposition}

\section{Continuity in $L^1$ with loss of decay}\label{contl1}

We consider an integral operator $A$  formally defined in \eqref{IO} and
with kernel $K(x,y)$  in \eqref{K}.
We assume
\begin{equation}\label{Phi}
\Phi\in  M^1(\bR^d),\quad 
\beta: (0,\infty)\to \bR.
\end{equation}
 Then, the kernel $K$  is well-defined for every $x, y \in \bR^d$. Indeed, since $M^1(\rd)\hookrightarrow L^1(\rd)$, the integral in \eqref{K} is  absolutely convergent. 
Inserting the kernel  expression \eqref{K} in the operator $A$, defined in  \eqref{IO}, and using the absolute convergence of the integrals we can apply Fubini's Theorem and infer
\begin{equation}\label{FIO2}
A f(x)=\intrdd   e^{- 2 \pi i  (\b(|u|)u \cdot  y-x\cdot u)}\Phi(u) f(y) \,dy \,du.
\end{equation}
\par
That is, the operator $A$ can be written in the form  of a Fourier integral operator  of type II. We recall that a { FIO of type II} with phase $\f$ and symbol $\sigma$  has the general form
\begin{equation}\label{type2}
T_{II,\f,\sigma}f(x)=\int_{\rdd}e^{-2\pi i (\f(y,u)-x\cdot u)}\sigma(y,u) f(y)dy\,d u
\end{equation}
hence
\begin{equation}
A=T_{II,\f,\sigma},\quad \text{with}\,\ \f(y,u)=\b(|u|)u \cdot  y\,\,\text{and}\,\,\, \sigma(y,u)=\Phi(u).
\end{equation}
FIO's of type II are the formal adjoints of FIO's of type I, defined in \eqref{FIO}.
Namely,
\begin{equation}\label{aggiunto}
(T_{I, \f,\sigma})^\ast=T_{II,\f,\sigma}.
\end{equation}

In general we do not expect that the integral operator $A$ in \eqref{IO} with kernel $K$ in \eqref{K} is continuous on $L^p(\rd)$, $1\leq p\leq \infty$, $p\not=2$. Indeed, we expect a loss of decay, as witnessed by the following example.

\begin{prop}\label{pro3.1}
	In dimension $d=1$, for any $1\leq p\leq 2, $ consider the weight function $$v_m(y)=(1+|y|)^{m},\quad y\in\bR,$$
	with $m\in\bR$ such that
	\begin{equation}
	m<\frac1p -\frac12.
	\end{equation}
Let $ \beta\in\cC^\infty((0,+\infty))$ 
	 such that   \begin{equation}\label{fi}
	\tilde{\f}(u)=\beta (|u|)u
	\end{equation}
	 extends to a nonlinear smooth diffeomorphism $\R\to\R$ satisfying
	\begin{equation}\label{tildaf}
	\tilde{\f}(u)=u,\quad |u|\geq 1.
	\end{equation}
	(hence, $\b(|u|)=1$, for $|u|\geq 1$). 
Let $\Phi\in \cC^\infty_0(\bR)$, $\Phi(u)=1$ for $|u|\leq 1$.\par
Then the operator $A$ in \eqref{FIO2} does not extend to a bounded operator from $L^p_{v_{m}}(\bR)$ to $L^p(\bR)$.
\end{prop}
\begin{proof} 	\emph{Step 1: Rephrasing the thesis}. Since $v_m(y)=(1+|y|)^{m}$ is a weight equivalent to
	$w_m(y):=\la y\ra ^m=(1+y^2)^{m/2}$, we can work with $w_m$ in place of $v_m$. 
	Since $A$ can be written as a type II Fourier integral operator, this amounts to considering the continuity from $L^p_{w_{m}}(\bR)$ to $L^p(\bR)$ of the operator $T_{II,\f,\sigma}$ in \eqref{type2}
	with $\f(y,u)=\tilde{\f}(u)y$ and symbol 
	$\sigma(y,u)=\Phi(u)$,
	with $\tilde{\f}$ and  $\Phi$ as in the statement.\par\medskip
	\emph{Step 2: From type II FIOs to type I FIOs.} By duality, the continuity of $T_{II,\f,\sigma}$ is equivalent to the boundedness of the adjoint $(T_{II,\f,\sigma})^*=  T_{I,\f,\sigma}$, 
	from $L^p(\bR)$ to $L^p_{w_{-m}}(\bR)$, for $2\leq p\leq \infty$.\par\medskip
	\emph{Step 3: Results for type I FIOs}. The continuity of $T_{I,\f,\sigma}$ 	from $L^p(\bR)$ to $L^p_{w_{-m}}(\bR)$ is equivalent to the boundedness 
	of the operator $w_{-m}T_{I,\f,\sigma}$ on $L^p(\bR)$.
	Now, observe that 
	\begin{align*}
	\la x \ra ^{-m}T_{I,\f,\sigma}f(x)&= \la x \ra ^{-m}\int e^{2\pi
		i\f(x,u)}
	\sigma(x,u)\hat{f}(u)\,du\\&=\int e^{2\pi
		i\f(x,u)}
	\la x \ra ^{-m}\sigma(x,u)\hat{f}(u)\,du\\
	&=\int e^{2\pi
		i\f(x,u)}
\tilde{\sigma}(x,u)\hat{f}(u)\,du:=T_{I,\f,\tilde{\sigma}}
	\end{align*}
with $\tilde{\sigma}(x,u)=\la x \ra ^{-m} \sigma(x,u)=\la x \ra ^{-m}\Phi(u).$\par
Now the type I FIO $T_{I,\f,\tilde{\sigma}}$ is not bounded on $L^p$, $2\leq p\leq\infty$, by Proposition \ref{casolp}.

\end{proof}

Continuity in weghted $L^1$ spaces, i.e.\ with a loss of decay, for the operator $A$ in \eqref{FIO2} can be proved by a Schur-type estimate for the kernel $K$. The following result addresses such estimates and Corollary \ref{cor:L1_cont} the corresponding continuity result. 
\par
\begin{theorem}\label{the:K_bounded}
Consider functions $\Phi\in M^1(\rd)$ and $\beta:(0,+\infty)\to\R$. Moreover, assume that for some exponent  $\g \in (-1,1]$, with $l=\lfloor d/2\rfloor +1$,
\begin{equation}\label{eq_u<1}
    \left|\partial^{\a}\b(|u|)u\right|\leq C_{\a}|u|^{\g +1-|\a|},\quad \text{ for }\,\,0\not=|u|\leq 1,\,\, |\a|\leq l,
\end{equation}
where $C_{\a}>0$, and 
\begin{equation}\label{eq_u>1}
    \left|\partial^{\a}\b(|u|)u\right|\leq C'_{\a},\quad\quad \text{ for }\,\, |u|\geq 1,\,\, 2\leq |\a|\leq  2l,
\end{equation}
with  $C'_{\a}>0$. Then the integral kernel in \eqref{K} satisfies
\begin{equation}\label{E1}
\int_{\bR^d}|K(x,y)|dx \leq C(1+|y|)^{{d}/(\g +1)},
\end{equation}
for a suitable constant $C>0$ independent of $y$.
\end{theorem}

\begin{proof}
To infer the estimate in \eqref{E1}, we write
$$\| K(\cdot,y)\|_{L^1} =\| \cF^{-1}(\Phi e^{-2\pi i\f(y,\cdot)})\|_{L^1}=\|\Phi e^{-2\pi i\f(y,\cdot)}\|_{\cF L^1},$$
where the phase is  $\f(y,u)=\beta(|u|)u\cdot y$. \par
We are going to show that
$$ e^{-2\pi i \f(y,\cdot)}\in W(\cF L^1, L^{\infty})(\bR^d),$$
with
\begin{equation}\label{f1}
\| e^{-2\pi i \f(y,\cdot)}\|_{W(\cF L^1, L^{\infty})}\leq C (1+|y|)^{d/(\gamma+1)};
\end{equation}
then the algebra property of $W(\cF L^1, L^{\infty})$ for $p=1$ in \eqref{ap}
yields
$$\|\Phi e^{-2\pi i\f(y,\cdot)}\|_{\cF L^1}\lesssim \|\Phi e^{-2\pi i\f(y,\cdot)}\|_{M^1}\lesssim \| e^{-2\pi i\f(y,\cdot)}\|_{W(\cF L^1, L^{\infty})}\|\Phi\|_{M^1}
$$
and this gives the claim \eqref{E1}. \par
We study the cases $|y|\leq 1$ and $|y|\geq1$ separately. The more difficult point is  $|y|\geq 1$, which is proved as follows. Using the dilation properties of $W(\cF L^1, L^{\infty})$ in Lemma \ref{l3}, the estimate \eqref{f1} follows from
\begin{equation}\label{l23}  e^{-2 \pi i\f (y,|y|^{-1/(\g +1)} u)}\in W(\cF L^1, L^{\infty}),
\end{equation}
uniformly with respect to $y$. Indeed, 
\begin{align*}
\left\|e^{-2 \pi i  \b\left(|u|\right)u \cdot y}\right\|_{W(\cF L^1, L^{\infty})} &\lesssim |y|^{d/(\gamma+1)}\left\|e^{-2 \pi i  \b\left(\frac{|u|}{|y|^{1/(\g +1)}}\right)\frac{u}{|y|^{1/(\gamma +1)}} \cdot y}\right\|_{W(\cF L^1, L^{\infty})} \\
&\lesssim (1+|y|)^{d/(\gamma+1)}\left\|e^{-2 \pi i  \b\left(\frac{|u|}{|y|^{1/(\gamma +1)}}\right)\frac{u}{|y|^{1/(\gamma +1)}} \cdot y}\right\|_{W(\cF L^1, L^{\infty})}.
\end{align*}
 The key idea is to use Lemmas \ref{l1} and \ref{l2} with $\mu(u)$ being 
 the phase $\f(y,u)=\beta(|u|)u\cdot y$ now rescaled by $|y|^{-1/(\g +1)}$: set
\begin{equation}
B_y(u):= 2 \pi \b\left(\frac{|u|}{|y|^{1/(\g +1)}}\right)\frac{u}{|y|^{1/(\g +1)}} \cdot y.
\end{equation}
We consider a test function $\chi \in \cC^{\infty}_{0}(\rd)$, such that $\chi(u)=1$ for $|u|\leq 4/3$ and $\chi(u)=0$ when $|u| \geq 5/3$ and write
$$e^{-i B_y(u)}=e^{-i B_y(u)}\chi(u)+ e^{-i B_y(u)}(1-\chi(u)).$$
We first show that 
\begin{equation*}
e^{-  i B_y} \chi\in W(\cF L^1, L^{\infty})(\bR^d),
\end{equation*}
uniformly with respect to $y$, that is, since $\chi$ is compactly supported, 
\begin{equation}\label{fasel1}
e^{-  i B_y} \chi\in \cF L^1(\bR^d).
\end{equation}
Using Lemma \ref{l1}, it is  enough to verify that the phase $B_y$ satisfies the estimates
\begin{equation}\label{e2}
|\partial^{\a}B_y(u)|\leq C_{\a}|u|^{\g +1-|\a|}
\end{equation}
for $|u|\leq 2$ (which is a neighborhood of the support of $\chi$) and $|\a|\leq \lfloor d/2\rfloor +1$. 

Observe that the estimates in the assumption \eqref{eq_u<1} actually hold for, say, $|u|\leq 2$, because \eqref{eq_u>1} implies \eqref{eq_u<1} for $1\leq |u|\leq 2$. Hence we have 
\begin{equation}\label{eq_u<y}
    |\partial^{\a}B_y(u)|\leq \frac{2 \pi C_{\a}}{|y|^{|\a|/(\g +1)}}\left( \frac{|u|}{|y|^{1/(\g +1)}}\right)^{\g+1-|\a|} |y|= 2 \pi C_{\a} |u|^{\g+1-|\a|}
\end{equation}
for $|u| \leq 2|y|^{1/(\g +1)} $, hence for $|u|\leq 2$. Since $\gamma+1>0$, by Lemma \ref{l1} we have \eqref{fasel1}.\\
We now  use Lemma \ref{l2} to show that
\begin{equation*}
e^{-  i B_y}(1-\chi) \in W(\cF L^1, L^{\infty})(\bR^d).
\end{equation*} 
It is sufficient to verify that the phase satisfies 
\begin{equation}\label{i1}
|\partial^{\a}B_y(u)|\leq C_{\a}
\end{equation}
for $|u|\geq 1$ and $2\leq |\a|\leq 2l$, with $l=\lfloor d/2 \rfloor +1$. If $1\leq |u|\leq |y|^{1/(\gamma +1)}$ the estimate \eqref{i1} follows by \eqref{eq_u<y}, because $\gamma\leq 1$. On the other hand, if $|u|\geq |y|^{1/(\gamma +1)}$ we use the hypothesis \eqref{eq_u>1}:
$$
|\partial^{\a}B_y(u)|\leq  \frac{2 \pi C'_{\a}}{|y|^{|\a|/(\gamma +1)}} |y|= \frac{2 \pi C'_{\a}}{|y|^{(|\a|-\gamma -1)/(\gamma +1)}} \leq {C}_{\a}''$$
for $2\leq |\a|\leq 2l$, because $\gamma \leq 1$. Then, Lemma \ref{l2} gives the claim.\par
It remains the case $|y|<1$. Here  we can argue as above without the dilation factor $|y|^{-1/(\g+1)}$ and repeating the same pattern we obtain the claim.
\end{proof}

In the previous result the weakest growth is reached when $\gamma=1$, the exponent in \eqref{E1} in that case being $d/2$. That growth is the same obtained even for smooth phase, as proved in the following result, and cannot be further reduced, as shown in Proposition \ref{pro3.1}. 
\begin{corollary}
Consider functions $\Phi\in M^1(\rd)$ and $\beta:(0,+\infty)\to\R$. Moreover, setting $l=\lfloor d/2\rfloor +1$, assume that the function $\b(|u|)$ extends to a $\cC^{2l}$ function on $\rd$ and satisfies
\begin{equation}\label{eq_u}
    \left|\partial^{\a}\b(|u|)u\right|\leq C_{\a},\quad \text{ for } u\in\R^d \,\text{ and }\,\,2\leq  |\a|\leq 2l.
\end{equation}
 Then, the integral kernel in \eqref{K} satisfies 
\begin{equation}
\int_{\R^d}|K(x,y)|dx \leq C(1+|y|)^{\frac{d}{2}}.
\end{equation}
\end{corollary}
\begin{proof}
The proof uses the same arguments as in Theorem \ref{the:K_bounded}. We split into the cases $|y|\geq 1$ and $|y|< 1$. We study first $|y|\geq 1$
and  prove that 
\begin{equation}\label{i2}
e^{- 2 \pi i  \b\left(\frac{|u|}{|y|^{1/2}}\right)\frac{u}{|y|^{1/2}} \cdot y} \in W(\cF L^1, L^{\infty})(\R^d)
\end{equation}
uniformly with respect to $y$.
Using Lemma \ref{l2}, we are reduced  to verify that the rescaled phase$$ B_y(u):= 2 \pi \b\left(\frac{|u|}{|y|^{1/2}}\right)\frac{u}{|y|^{1/2}} \cdot y$$ satisfies the estimate 
\begin{equation*}
|\partial^{\a}B_y(u)|\leq C_{\a},\quad \forall u\in\rd
\end{equation*}
 and $2\leq |\a|\leq 2l$. By the hypothesis \eqref{eq_u},
$$
|\partial^{\a}B_y(u)|\leq C_{\a} \frac{2 \pi}{|y|^{|\a|/2}}\cdot |y|= {C_{\a}'} \frac1{|y|^{(|\a|-2)/2}} \leq {C_{\a}^{''}}$$
since $|\a|\geq 2$ and $|y|\geq 1$; this gives \eqref{i2}.\par
Then, by Lemma \ref{l3}, we have
\begin{equation*}
    \left\|e^{-2 \pi i  \b\left(|u|\right)u \cdot y}\right\|_{W(\cF L^1, L^{\infty})} 
    \lesssim (1+|y|)^{d/2}\left\|e^{-2 \pi i  \b\left(\frac{|u|}{|y|^{1/2}}\right)\frac{u}{|y|^{1/2}} \cdot y}\right\|_{W(\cF L^1, L^{\infty})}.
\end{equation*}
As $\p \in M^1(\R^d)$, we have that $\Phi(u) e^{- i  \b(|u|)u \cdot y} \in M^1(\R^d)$. Moreover, we can write 
\[
K(x,y)=\cF^{-1}\left[\Phi(u) e^{- i  \b(|u|)u \cdot y)}\right](x).
\]
Hence,
\begin{align*}
    \int_{\R^d}|K(x,y)|dx & = \left\|\cF^{-1}\left[\Phi(u) e^{- i  \b(|u|)u \cdot y)}\right]\right\|_{L^1}\lesssim \left\|\cF^{-1}\left[\Phi(u) e^{- i  \b(|u|)u \cdot y)}\right]\right\|_{M^1}\\
    &\lesssim \left\|\Phi(u) e^{- i  \b(|u|)u \cdot y)}\right\|_{M^1}
    \lesssim \|\Phi\|_{M^1} \left\|e^{-2 \pi i  \b\left(|u|\right)u \cdot y}\right\|_{W(\cF L^1, L^{\infty})} \\
    &\leq C(1+|y|)^{\frac{d}{2}},
\end{align*}
for some $C>0$.

The case $|y|<1$ is attained with the same pattern above, without the dilation factor ${|y|^{-\frac1{2}}}$.
\end{proof}
\begin{corollary}\label{cor:bound_K}
Consider functions $\Phi\in M^1(\rd)$ and $\beta:(0,+\infty)\to\R$. Assume that for some $\g \in (-1,1]$ and $a \in \bR$, 
\begin{equation}\label{betatilde}
\widetilde{\b}:= \b-a
\end{equation} satisfies, with $l=\lfloor d/2\rfloor +1$,
\begin{equation}\label{i4}
    \left|\partial^{\a}\widetilde{\b}(|u|)u\right|\leq C_{\a}|u|^{\g +1-|\a|},\quad \text{ for }\,\,|u|\leq 1,\,\, |\a|\leq l,
\end{equation}
for $C_{\a}>0$, and 
\begin{equation}\label{i5}
    \left|\partial^{\a}\widetilde{\b}(|u|)u\right|\leq C'_{\a},\quad\quad \text{ for }\,\, |u|\geq 1,\,\, 2\leq |\a|\leq  2l,
\end{equation}
with  $C'_{\a}>0$.  Then the integral kernel in \eqref{K} satisfies
\begin{equation}\label{i6}
\int_{\R^d}|K(x,y)|dx \leq C(1+|y|)^{{d}/(\g +1)},
\end{equation}
for a suitable constant $C>0$, independent of the variable $y$.
\end{corollary}

\begin{proof}
By the proof of Theorem \ref{the:K_bounded} we know that $e^{-2 \pi i  \widetilde{\b}(|u|)u \cdot y} \in W(\cF L^1, L^{\infty})(\rd) $ with $$\left\|e^{-2 \pi i  \widetilde{\b}\left(|u|\right)u \cdot y}\right\|_{W(\cF L^1, L^{\infty})} \lesssim (1+|y|)^{d/(\g+1)}.$$
By \eqref{betatilde}, 
$$e^{-2 \pi i  \b(|u|)u \cdot y}= e^{-2 \pi i  a u \cdot y}\cdot e^{- 2 \pi i  \widetilde{\b}(|u|)u \cdot y}=M_{-ay} e^{- 2 \pi i  \widetilde{\b}(|u|)u \cdot y}.$$
Using the invariance property of $W(\cF L^1, L^{\infty})(\R^d)$ with respect to time-frequency shifts in \eqref{tfwiener}:
\begin{align*}
\left\|e^{-2 \pi i {\b}\left(|u|\right)u \cdot y}\right\|_{W(\cF L^1, L^{\infty})} &=
\left\|M_{-ay} e^{- 2 \pi i  \widetilde{\b}(|u|)u \cdot y}\right\|_{W(\cF L^1, L^{\infty})}\\&
\asymp \left\|e^{-2 \pi i  \widetilde{\b}\left(|u|\right)u \cdot y}\right\|_{W(\cF L^1, L^{\infty})} \\&
\leq C (1+|y|)^{d/(\g+1)}. 
\end{align*}
This concludes the proof.
\end{proof}

We end up this section by using the previous results for the integral kernel $K(x,y)$ to attain the $L^1$-boundedness for the corresponding operator $A$. The cost is a loss of decay, as explained below.
\begin{corollary}\label{cor:L1_cont}
Assume the hypotheses of Corollary \ref{cor:bound_K} and consider the weight function 
$$v(y)= (1+ |y|)^{d/(\g+1)}.$$
 Then the integral operator $A$ in \eqref{FIO2} with kernel $K$ in \eqref{K} is bounded from  $L^1_{v}(\rd)$ into $ L^1(\rd)$.
\end{corollary}
\begin{proof}
By Corollary \ref{cor:bound_K}, we know that the kernel $K(x,y)$  satisfies the estimate in \eqref{i6}.
 Let $f \in L^1_{v}(\rd)$; using Fubini's Theorem and the estimate in \eqref{i6},
\begin{align*}
    \|Af(x)\|_{L^1} &= \int_{\R^d}|Af(x)|dx = \int_{\R^d}\left|\int_{\R^d}K(x,y)f(y)dy\right|dx  \\
    &\leq \int_{\R^d}\int_{\R^d}|K(x,y)||f(y)|dydx = \int_{\R^d}|f(y)|\left(\int_{\R^d}|K(x,y)|dx\right) dy \\
    &\leq \int_{\R^d}|f(y)|C(1+|y|)^{{d}/(\g +1)} dy = C \|f\|_{L^1_v},
\end{align*}
as desired.
\end{proof}
\section{Continuity in $L^2$}\label{contl2}
A natural question is whether the assumptions of Corollary \ref{cor:bound_K}, that give continuity  of the operator $A$ on  $L^1(\rd)$ with a loss of decay, guarantee at least continuity of  $A$ on $L^2(\rd)$ without any loss. The answer is negative even in dimension $d=1$, as shown by the following result.
\begin{prop}\label{41}
Let $d=1$. There exists an operator $A$ as in \eqref{FIO2}, with $\beta$ and $\Phi$ satisfying the assumptions of Corollary \ref{cor:bound_K}, that is not bounded on $L^2(\rd)$.
\end{prop}
\begin{proof}
Consider a  function $h(r)$ such that $\Phi(u)=h(|u|)\in \cC^\infty_0(\R)$, and $h(0)\not=0$. For $\gamma\in (0,1)$, set $\beta(u)=u^\gamma$.
Finally,  take $\chi \in \cC^{\infty}_0(\R)$ such that $ \chi(u)= 1 $ when $u \in \supp \Phi$, and consider the function $\tilde{\b}(u)= \chi(u)\b(u)$. Consider the operator $A$ with integral kernel
 $$
K(x,y)= \int_{\R}h(|u|) e^{- 2 \pi i  (u  x+ \tilde{\b}(|u|)u y)}du =\int_{\R}h(|u|) e^{- 2 \pi i  (u  x+ \b(|u|)u  y)}du.
$$
We now show that $A$ is not bounded on $L^2(\rd)$.\par
 For  $f \in\cS(\bR)$,
\begin{align*}
    Af(x)&= \int_{\R}K(x,y)f(y)dy= \int_{\R}\int_{\R}h(|u|) e^{- 2 \pi i  (u \cdot x+ \b(|u|)u \cdot  y)} f(y) du dy \\
    &=\int_{\R}h(|u|) e^{- 2 \pi i  u \cdot x} \left(\int_{\R} f(y)e^{- 2 \pi i  \b(|u|)u\cdot y} dy\right) du =\\
    &=\int_{\R}h(|u|) e^{- 2 \pi i  u \cdot x} \hat{f}( \b(|u|)u ) du = \cF\left[h(|u|) \hat{f}( \b(|u|)u )\right](x).
\end{align*}
Then, by Parseval's Theorem,
$$
    \|Af\|^2_2 = \left\|h(|u|) \hat{f}( \b(|u|)u )\right\|^2_2 = \int_{\R}|h(|u|)|^2 |\hat{f}( \b(|u|)u ) |^2 dx.
$$
We perform the change of variable 
$$\tilde{u}= \b(|u|)u= |u|^{\gamma}u = \left\{ \begin{array}{ll}
u^{\gamma+1}, & u\geq 0,\\
-|u|^{\gamma+1}, & u<0,
\end{array}\right.,$$
so that
$$u = \left\{ \begin{array}{ll}
\tilde{u}^{\frac1{\gamma+1}}, & \tilde{u}\geq 0,\\
-(-\tilde{u})^{\frac1{\gamma+1}}, & \tilde{u}<0.
\end{array}\right.$$
and $du = \frac1{1+\gamma} |\tilde{u}|^{\frac1{1+\gamma}-1} d\tilde{u}$. In this way, we obtain 
\begin{equation*}
    \|Af\|^2_2 = \int_{\R}|h(|u|)|^2 |\hat{f}( \b(|u|)u ) |^2 du 
    =\frac1{1+\gamma} \int_{\R} |\tilde{u}|^{\frac1{1+\gamma}-1}|h(|\tilde{u}|^{\frac1{1+\gamma}})|^2 |\hat{f}(\tilde{u}) |^2 d\tilde{u}.
\end{equation*}
Now,  the last expression is controlled by  $C \|f\|^2_{L^2}$, for a suitable constant $C>0$ and for every $f \in \cS(\R)$, if and only if $$|\tilde{u}|^{-\frac{\gamma}{1+\gamma}}|h(|\tilde{u}|^{\frac1{1+\gamma}})|^2 \in L^{\infty}(\R),$$
(notice that $h(|\tilde{u}|^{\frac1{1+\gamma}})$ has compact support) and this fails since $-\gamma/(1+\gamma) < 0$ and $|h(|u|)|\geq \delta>0$ in a neighborhood of $0$.
\end{proof}

We now look for suitable assumptions on the functions $\Phi$ and $\beta$ which guarantee $L^2$-continuity of the operator $A$. A successful choice is shown below.
\begin{theorem}\label{tl2cont}
Consider $\Phi \in L^{\infty}(\R^d)\cap L^1(\rd)$. Let $\b: (0,\infty) \to \R$ satisfy the following assumptions:\\
(i) $\b\in \cC^1((0,\infty))$;\\
(ii) There exists  $\delta>0$ such that $\b(r)\geq \delta$, for all $r >0$;\\
(iii) There exist $B_1, B_2 >0$, such that \begin{equation}\label{ei1}B_1 \leq \frac{d}{dr}(\b(r)r) \leq B_2, \quad \forall \,r >0.\end{equation}
Then the integral operator $A$ with kernel $K$ in \eqref{K} is bounded on $L^2(\R^d)$. 
\end{theorem}
\begin{proof}
We first observe that, since $\Phi \in  L^1(\rd)$, the integral defining the kernel $K(x,y)$ is absolutely convergent and $K$ is well-defined. 
Let $f \in \cS(\R^d)$, using Fubini's Theorem we can write
\begin{align*}
    Af(x)&= \int_{\R^d}K(x,y)f(y)dy= \int_{\R^d}\int_{\R^d}\Phi(u) e^{- 2 \pi i  (\b(|u|)u \cdot y-u \cdot x)} f(y) du dy \\
    &=\int_{\R^d}\Phi(u) e^{2 \pi i  u \cdot x} \left(\int_{\R^d} f(y)e^{- 2 \pi i  \b(|u|)u \cdot y} dy\right) du \\
    &=\int_{\R^d}\Phi(u) e^{2 \pi i  u \cdot x} \hat{f}( \b(|u|)u ) du = \cF^{-1}\left[\Phi(u) \hat{f}( \b(|u|)u )\right](x).
\end{align*}
Then, by Parseval's Theorem,
$$
    \|Af\|^2_2 = \|\Phi(u) \hat{f}( \b(|u|)u )\|^2_2=  \int_{\R^d}|\Phi(u)|^2 |\hat{f}( \b(|u|)u ) |^2 du.
$$
Changing to polar coordinates $u = r  \theta$, with $r >0$ and $\theta \in \mathbb{S}^{d-1}$, we have  $du= r^{d-1}drd\theta$ and 
\begin{align*}
    \|Af\|^2_2 &= \int_{\R^d}|\Phi(u)|^2 |\hat{f}( \b(|u|)u ) |^2 du\\
    &=\int_0^{\infty}\int_{S^{d-1}}|\Phi(r \theta)|^2 |\hat{f}( \b(r)r \theta ) |^2 r^{d-1}d\theta dr.
\end{align*}
Observe that the function $\varphi(r):=\b(r)r$ is strictly increasing by assumption $(iii)$. Performing the change of variable $\tilde{r}= \varphi(r)$, as $r\in (0,\infty)$, there exists an $a \geq 0$ such that $\varphi((0,+\infty))= (a,\infty)\subseteq (0,\infty)$. Moreover, $\varphi(r)$ has an inverse $\varphi^{-1}(\tilde{r})=r$ such that $$B_2^{-1} \leq \frac{d}{d\tilde{r}}(\varphi^{-1}(\tilde{r})) \leq B_1^{-1} \ \ \ \ \text{ for all }\tilde{r} >0.$$
Further, by assumption $(ii)$,
$$\frac1{\delta}\geq \frac1{\b(r)}= \frac{r}{\b(r)r}= \frac{\varphi^{-1}(\tilde{r})}{\tilde{r}}. $$
Then we can write,
\begin{align*}
    \|Af\|^2_2 &= \int_0^{\infty}\int_{S^{d-1}}|\Phi(r \theta)|^2 |\hat{f}( \b(r)r \theta ) |^2 r^{d-1}d\theta dr  \\
    &=\int_0^{\infty}\int_{S^{d-1}}|\Phi(\varphi^{-1}(\tilde{r}) \theta)|^2 |\hat{f}( \tilde{r} \theta ) |^2 (\varphi^{-1}(\tilde{r}))^{d-1}\frac{d}{d \tilde{r}}(\varphi^{-1}(\tilde{r}))d\theta d\tilde{r} \\
    &\leq \sup_{r,\theta}\left\{|\Phi(\varphi^{-1}(\tilde{r}) \theta)|^2  \left(\frac{\varphi^{-1}(\tilde{r})}{\tilde{r}}\right)^{d-1} \frac{d}{d \tilde{r}}
		(\varphi^{-1}(\tilde{r}))\right\}\int_0^{\infty}\int_{S^{d-1}} |\hat{f}( \tilde{r} \theta ) |^2 (\tilde{r})^{d-1}d\theta d\tilde{r} \\
    &\leq\|\Phi\|_{\infty}^2  \left(\frac1{\delta}\right)^{d-1} B_1^{-1}\|f\|_2^2.
\end{align*}
This gives $\|Af\|_2\leq C\|f\|_2$, for every $f\in\cS(\rd)$. By density argument we obtain the claim for every $f\in\lrd$.
\end{proof}
\begin{remark}
The previous proof still works if we change  the function  $\b$ with $-\b$. Hence, under the assumptions of Theorem \ref{tl2cont}
with assumptions $(ii)$ and $(iii)$ replaced by:\\
(ii)' There exists  $\delta<0$ such that $\b(r)\leq \delta$, for all $r >0$;\\
(iii)' There exist $B_1, B_2 <0$, such that \begin{equation*}B_1 \leq \frac{d}{dr}(\b(r)r) \leq B_2, \quad \forall \,r >0;\end{equation*} 
 the integral operator $A$ with kernel $K$ in \eqref{K} is bounded on $L^2(\rd)$. 
\end{remark}

We now exhibit a class of examples of functions $\Phi\in M^1(\rd)$, hence fulfilling the assumptions of Theorem \ref{the:K_bounded} and related corollaries, as well as those of Theorem \ref{tl2cont}, which are of special interest in the study of Boltzmann equation, cf.\ \cite{Lod2010}. 

\begin{example}\label{exm1}
Consider the function 
 \begin{equation}\label{PM1}
\Phi(u)=\frac{|u|}{(1+|u|^2)^m},\quad \mbox{for}\,\,\,m > \frac{d+1}{2}.
\end{equation}
Then $\Phi\in  M^1(\rd)$. (Observe that $\Phi(u)=h(|u|)$).
\end{example}
\begin{proof}
We consider a function $\chi \in \cC^{\infty}_{0}(\rd)$, such that $\chi(u)=1$ when $|u|\leq1/{2}$ and $\chi(u)=0$ when $|u| \geq 1$. We write
$$\Phi(u)=\Phi(u)\chi(u)+\Phi(u)(1- \chi(u))
$$
and show that
\begin{equation}\label{eq_phichi}
\Phi(u) \chi(u) \in M^1(\R^d)
\end{equation}
and
\begin{equation}\label{eq_phi1-chi}
\Phi(u)(1- \chi(u)) \in M^1(\R^d).
\end{equation}

To prove \eqref{eq_phichi}, we choose another cut-off function  $\widetilde{\chi} \in \cC^{\infty}_{0}(\rd)$ such that $$\widetilde{\chi}(u)=1 \quad \textrm{for}\,\,u \in \mbox{supp}\,\chi;$$ then $\widetilde{\chi}\cdot \chi= \chi$. Consider now the function $h(u)=|u|$, which  is in $\cC^\infty(\rd\setminus\{0\})$ and positively homogeneous of degree $1$ and set $f=h\chi$.   Lemma \ref{l4} gives, for $\psi\in \cS(\rd)$,
$$|V_\psi f(x,\xi)|\leq C (1+|\xi|)^{-(d+1)}
$$
hence, by \eqref{normwiener},
$$\|f\|_{W(\cF L^1, L^{\infty})}=\| \|V_\psi f(x,\cdot)\|_{L^1}\|_{L^\infty}<\infty,$$
that is  $|u|\chi \in W(\cF L^1, L^{\infty})(\R^d)$. Since $$\frac{\widetilde{\chi}(u)}{(1+|u|^2)^m} \in \cS(\rd) \subseteq M^1(\rd).$$ we can write $$\Phi(u)\chi(u) = |u|\chi(u) \cdot \frac{\widetilde{\chi}(u)}{(1+|u|^2)^m} \in M^1(\rd), $$
by Proposition \ref{p3}.\par
Finally,  to show \eqref{eq_phi1-chi}, we observe that $\Phi(u)(1- \chi(u)) = 0$ for $|u|\leq 1/2$, 
hence the singularity at the origin is removed and  $\Phi(u)(1-\chi(u))\in W^{k,1}(\rd)$ for all $k\in\bN$, provided that $2m-1> d$. We then choose $k> d$ and  apply the inclusion relations between the Potential Sobolev space $W^{k,1}(\rd)$ and the Feichtinger's algebra $M^1(\rd)$ in Lemma \ref{l5},  which gives  \eqref{eq_phi1-chi}.
\end{proof}

\section{Acknowledgments} The authors are
grateful to Professors Ricardo J. Alonso  and  Bertrand Lods for proposing this study and for his valuable suggestions. The first two authors were partially supported by  the Gruppo Nazionale per l'Analisi Matematica, la Probabilit\`a e le loro Applicazioni (GNAMPA) of the Istituto Nazionale di Alta Matematica (INdAM). The third author was partially supported by the  Generalitat Valenciana (Project VALi+d Pre Orden 64/2014 and Orden 86/2016)
\bibliographystyle{amsalpha}

\bibliography{biblio_notes}
 
\end{document}